\documentclass[11pt]{amsart}
\usepackage{amsmath}
\usepackage{amssymb}
\usepackage{graphicx}
\usepackage{url}
\usepackage{multirow}
\usepackage{comment}
\usepackage{algorithm}
\usepackage{algorithmic}
\usepackage{float}
\usepackage{comment}
\usepackage[font=small,labelfont=bf]{caption}
\usepackage{subcaption}
\usepackage{color}
\usepackage{hyperref}

\newtheorem{theorem}{Theorem}[section]
\newtheorem{corollary}[theorem]{Corollary}
\newtheorem{lemma}[theorem]{Lemma}

\theoremstyle{definition}
\newtheorem{definition}[theorem]{Definition}
\newtheorem{example}[theorem]{Example}
\newtheorem{remark}[theorem]{Remark}

\numberwithin{equation}{section}

\title[Mirror Descent-Type Methods with Weighting Scheme for VIs] {Mirror Descent Methods with Weighting Scheme for Outputs for Constrained Variational Inequality Problems}

\author[M.~S.~Alkousa]{Mohammad S. Alkousa}
\address[M.~S.~Alkousa]{Innopolis University, and Moscow Institute of Physics and Technology, Russia.}
\email{\tt m.alkousa@innopolis.ru}

\author[B.~A.~Alashqar]{Belal A. Alashqar}
\address[B.~A.~Alashqar]{Moscow Institute of Physics and Technology, Russia.}
\email{\tt alashkar.ba@phystech.edu}

\author[F.~S.~Stonyakin]{Fedor~S.~Stonyakin}
\address[F.~S.~Stonyakin]{Moscow Institute of Physics and Technology, and V. I. Vernadsky Crimean Federal University, Russia.}
\email{{\tt fedyor@mail.ru}}

\author[T.~Nabhani]{Tarek Nabhani}
\address[T.~Nabhani]{Faculty of Science and Humanities, Shaqra University, Al-Dawadmi, Saudi Arabia}
\email{\tt tnabhani@su.edu.sa}

\author[S.~S.~Ablaev]{Seydamet S. Ablaev}
\address[S.~S.~Ablaev]{V. I. Vernadsky Crimean Federal University, Russia.}
\email{\tt seydamet.ablaev@yandex.ru}


\begin{document}

\begin{abstract}
This paper is devoted to the variational inequality problems. We consider two classes of problems, the first is classical constrained variational inequality, and the second is the same problem with functional (inequality type) constraints. To solve these problems, we propose mirror descent-type methods with a weighting scheme for the generated points in each iteration of the algorithms. This scheme assigns smaller weights to the initial points and larger weights to the most recent points, thus it improves the convergence rate of the proposed methods. For the variational inequality problem with functional constraints, the proposed method switches between productive and non-productive steps in the dependence on the values of the functional constraints at iterations. We analyze the proposed methods for the time-varying step sizes and prove the optimal convergence rate for variational inequality problems with bounded and monotone operators.  The results of numerical experiments of the proposed methods for classical constrained variational inequality problems show a significant improvement over the modified projection method.
\end{abstract}

\maketitle

\section{Introduction}
Variational inequalities (VIs) cover as a special case many optimization problems such as minimization problems, saddle point problems, and fixed point problems (see Examples \ref{ex:minproblem}, \ref{ex:saddleproblem} and \ref{ex:fixedproblem}, below). They often arise in various mathematical problems, such as optimal control, partial differential equations, mechanics, finance, and many others. They play a key role in solving equilibrium and complementarity problems \cite{facchinei2003finite,harker1990finite}, in machine learning research such as generative adversarial networks \cite{goodfellow2020generative}, supervised/unsupervised learning \cite{joachims2005support,xu2004maximum}, reinforcement learning \cite{jin2020efficiently,omidshafiei2017deep}, adversarial training \cite{madry2017towards}, and generative models \cite{daskalakis2017training,gidel2018variational}.

Numerous researchers have dedicated their efforts to exploring theoretical aspects related to the existence and stability of solutions and constructing iterative methods for solving the classical VIs (by classical, we mean the problems without functional ''inequality types'' constraints), see \eqref{main_constrained_prob}. A significant contribution to the development of numerical methods for solving the classical VIs was made in the 1970’s, when the extragradient method was proposed in \cite{korpelevich1976extragradient}. More recently, Nemirovski in his seminal work \cite{nemirovski2004prox} proposed a non-Euclidean variant of this method, called the Mirror Prox algorithm, which can be applied to Lipschitz continuous operators.   Different methods with similar complexity were also proposed in \cite{auslender2005interior,gasnikov2019adaptive,nesterov2007dual}. Besides that, in \cite{nesterov2007dual}, Nesterov proposed a method for variational inequalities with a bounded variation of the operator, i.e., with a non-smooth operator. There is also extensive literature on variations of extragradient method that avoid taking two steps or two gradient computations per iteration, and so on (see for example \cite{hsieh2019convergence,malitsky2020forward}). 

Another important class of VIs is the problem with functional constraints (inequality type), see \eqref{main_func_cons_1}. The presence of such type of the constraints makes these problems more difficult to solve. This class of problems arises in many fields of mathematics, among them are economic equilibrium models \cite{Levin1993Mathematical}, game theory \cite{Garcia1981Pathways}, constrained Markov potential games \cite{Alatur2023Provably,Jordan2024Independent}, generalized Nash equilibrium problems with jointly-convex constraints \cite{Facchinei2010Generalized,Jordan2023First} and hierarchical programming problems \cite{Migdalas1996Hierarchical}, in mathematical physics \cite{Baiocchi1988Variational}. See \cite{Antipin2000Solution} to see some examples. Also, this class of problems encompasses important applications in machine learning including reinforcement learning with safety constraints \cite{Tengyu2021}, and learning with fairness constraints  \cite{Lowy2022,Muhammad2019}.

For VIs with functional constraints, the previous works have focused on primal-dual algorithms based on the (augmented) Lagrangian function to handle the constraints and penalty methods \cite{Auslender1999Asymptotic,Auslender2003Variational,He2004modified,Zhu2003Augmented}. These algorithms and their convergence guarantees crucially depend on information about the optimal Lagrange multipliers. In \cite{Zhang2024Primal}, it was proposed a primal method without knowing any information on the optimal Lagrange multipliers, with proving its convergence rate for the problem with monotone operators under smooth constraints. In \cite{Yang2023Solving }, it was presented a first-order method (ACVI) which combines path-following interior point methods and primal-dual methods. In \cite{Chavdarova2024Primal}, the authors proposed a primal-dual approach to solve the VIs with general functional constraints by taking the last iteration of ACVI. Although there are many works for the VIs with functional constraints, they remain very few compared to the existing works for the classical constrained problem.

In this paper, we propose Mirror Descent type methods for solving the classical variational inequality problem, and the same problem with functional constraints (inequality types), see problems \eqref{main_constrained_prob} and \eqref{main_func_cons_1}. The mirror descent method, for minimization problems, originated in \cite{Nemirovskii1979efficient,Nemirovsky1983Complexity} and was later analyzed in \cite{Beck2003Mirror}, is considered as the non-Euclidean extension of standard subgradient methods. This method is used in many applications, see \cite{applications_tomography_2001,article:Nazin_2011,article:Nazin_2014,article:Nazin_2013} and references therein.  The standard subgradient methods employ the Euclidean distance function with a suitable step-size in the projection step. Mirror descent extends the standard projected subgradient methods by employing a nonlinear distance function with an optimal step-size in the nonlinear projection step \cite{article:luong_weighted_mirror_2016}. The Mirror Descent method not only generalizes the standard subgradient methods but also achieves a better convergence rate \cite{article:doan_2019}. It is also applicable to optimization problems in Banach spaces where gradient descent is not \cite{article:doan_2019}. An extension of the mirror descent method for constrained problems was proposed in \cite{article:beck_comirror_2010,Nemirovsky1983Complexity}. The class of non-smooth optimization problems with non-smooth functional constraints attracts widespread interest in many areas of modern large-scale optimization and its applications \cite{Nemirovski_Robust,Shpirko_primal}. In terms of continuous optimization with functional constraints, there is a long history of studies. The monographs in this area include \cite{24,94}. Some of the works on first-order methods for convex optimization problems with convex functional constraints include (for example, but not limited to) \cite{article:adaptive_mirror_2018,71,stonyakin2018,stonyakin2019,119} for the deterministic setting and \cite{3,2,66,126} for the stochastic setting. 

Recently in \cite{Zhu2024Convergence}, for the projected subgradient method, the optimal convergence rate was proved using the previously mentioned time-varying step size with a new weighting scheme for the generated points each iteration of the algorithm.  This convergence rate remains the same (optimal) even if we slightly increase the weight of the most recent points, thereby relaxing the ergodic sense. These results were recently extended to mirror descent methods for constrained minimization problems in \cite{Alkousa2024Optimal} and for minimization problems with functional constraints (inequality type) in \cite{Alkousa2024Mirror}. 

In this paper, for the classical constrained variational inequality problem, we propose a mirror descent-type method (Algorithm \ref{alg_mirror_descent}) with a weighting scheme for the points generated in each iteration of the algorithm. We extend Algorithm \ref{alg_mirror_descent} (see Algorithm \ref{alg:MD_func_cons}) to be applicable to the variational inequality problem with functional constraints by switching between productive and non-productive steps. We analyze the proposed methods for the time-varying step sizes and obtain the optimal convergence rate (for Algorithm \ref{alg_mirror_descent}) for the class of variational inequality problems with bounded and monotone operators. 

The paper consists of an introduction and five main sections. In Sect. \ref{sect_basics}  we mentioned the basic facts, definitions, and tools for variational inequalities.  Sect. \ref{sectinMirror} devoted to the classical constraint variational inequality problem. We proposed a mirror descent method (Algorithm \ref{alg_mirror_descent}) with a weighting scheme for the points generated in each iteration of the algorithm, we analyzed Algorithm \ref{alg_mirror_descent} and proved its optimal convergence rate for the class of variational inequality problems with bounded and monotone operators. In Sect. \ref{sectinAlgs}, we proposed an extension of Algorithm \ref{alg_mirror_descent} (see \ref{alg:MD_func_cons}) to solve a more complicated variational inequality problem with functional constraints. In Sect. \ref{sect_numerical} we present numerical experiments that demonstrate the efficiency of the proposed weighting scheme in Algorithm \ref{alg_mirror_descent}, and compare its work with a modified projection method, proposed in \cite{Khanh2014Modified}, to solve some examples of the variational inequality problem. In the last section \ref{sec_conclusions}, we review the results obtained in the paper.

\section{Fundamentals}\label{sect_basics}

Let $(\mathbf{E},\|\cdot\|)$ be a normed finite-dimensional vector space, with an arbitrary norm $\|\cdot\|$, and $\mathbf{E}^*$ be the conjugate space of $\mathbf{E}$ with the following norm
$$
    \|y\|_{*}=\max\limits_{x \in \mathbf{E}}\{\langle y,x\rangle: \|x\|\leq1\},
$$
where $\langle y,x\rangle$ is the value of the continuous linear functional $y \in \mathbf{E}^*$ at $x \in \mathbf{E}$.

Let $Q \subset \mathbf{E}^n$ be a convex compact set with a diameter $D >0$, i.e., $\max_{x,y \in Q}\|x-y\| \leq D$,  and $\psi: Q \longrightarrow \mathbb{R}$ be a proper closed differentiable and $\sigma$-strongly convex (called prox-function or distance generating function). The corresponding Bregman divergence is defined as 
$$
    V_{\psi} (x, y) = \psi (x) - \psi (y) - \langle \nabla \psi (y), x - y \rangle, \quad \forall x , y \in Q. 
$$
For the Bregmann divergence, it holds the following inequality
\begin{equation}\label{eq_breg}
    V_{\psi}(x, y) \geq \frac{\sigma}{2} \|y - x\|^2, \quad \forall x, y \in Q. 
\end{equation}

\begin{definition}($\delta$-monotone operator). 
Let $\delta > 0$. The operator $F :Q \longrightarrow \textbf{E}^*$ is called $\delta$-monotone, if it holds
\begin{equation}\label{d_monot}
    \langle F(y)-F(x),y-x\rangle \geq -\delta, \quad \forall x, y \in Q.
\end{equation}
\end{definition}
For example, we can consider $F = \nabla_{\delta} f$ for $\delta$-subgradient $\nabla_{\delta} f(x)$ of convex function $f$ at a point $x \in Q$: $f(y) - f(x) \geq \langle  \nabla_{\delta} f(x), y - x \rangle - \delta $ for each $y \in Q$ (see e.g.,  Chapter 5 in \cite{Polyak}). 

When $\delta = 0$, then the operator $F$ is called monotone, i.e.,  
\begin{equation}\label{eq:CondMonotone}
    \langle F(x) - F(y) , x - y \rangle \geq 0, \quad \forall x, y \in Q.
\end{equation}

We say that the operator $F$ is bounded on $Q$, if there exist $L_F >0$ such that 
\begin{equation}\label{bound_cond}
    \|F (x)\|_* \leq L_F, \quad \forall x \in Q. 
\end{equation}

The following identity, known as the three points identity, is essential in analyzing the mirror descent method.

\begin{lemma}\label{three_points_lemma}(Three points identity \cite{Chen1993}) 
Suppose that $\psi: \mathbf{E} \longrightarrow (- \infty, \infty]$ is a proper
closed, convex, and differentiable function over $\operatorname{dom}(\partial \psi)$. Let $a, b \in \operatorname{dom}(\partial(\psi))$ and $c \in \operatorname{dom} (\psi)$. Then it holds 
\begin{equation}\label{eq_three_points}
    \left\langle\nabla \psi(b)-\nabla \psi(a), c-a\right\rangle=V_{\psi}(c,a)+V_{\psi}(a, b)-V_{\psi}(c, b) .
\end{equation}
\end{lemma}

\noindent
\textbf{Fenchel-Young inequality}(\cite{Beck2003Mirror}). For any $b \in \mathbf{E}, a \in \mathbf{E}^*$, it holds the following inequality
\begin{equation}\label{Fenchel_Young_ineq}
     \langle  a , b \rangle \leq \frac{\|a\|_*^2}{2 \lambda} + \frac{\lambda\|b\|^2}{2}, \quad \forall \lambda > 0.
\end{equation}

\section{Mirror descent method for constrained variational inequality  problem}\label{sectinMirror}
In this section, we consider the following variational inequality problem
\begin{equation}\label{main_constrained_prob}
    \text{Find} \quad x^* \in Q  : \quad  \langle F(x), x^* - x \rangle \leq 0 \quad \forall x \in Q, 
\end{equation}
where $F: Q \longrightarrow \textbf{E}^*$ is a continuous, bounded (i.e., \eqref{bound_cond} holds), and $\delta$-monotone operator (i.e., \eqref{d_monot} holds). 

Under the assumption of continuity and monotonicity (i.e., $\delta = 0 $) of the operator $F$, the problem \eqref{main_constrained_prob} is equivalent to a Stampacchia \cite{giannessi1998minty} (or strong \cite{nesterov2007dual}) variational inequality, in which the goal is to find $x^* \in Q $ such that 
\begin{equation}\label{prob:VIstrong}
    \langle F(x^*), x^* - x \rangle \leq 0, \quad \forall x \in Q.
\end{equation}

To emphasize the extensiveness of the problem \eqref{main_constrained_prob} (or \eqref{prob:VIstrong}), we mention three common special cases for VIs.

\begin{example}[Minimization problem]\label{ex:minproblem}
Let us consider the minimization problem 
\begin{equation}\label{min_problem}
    \min_{x \in Q} f(x), 
\end{equation}
and assume that $F(x) = \nabla f(x)$, where $\nabla f(x)$ denotes the (sub)gradient of $f$ at $x$. Then, if $f$ is convex, it can be proved that $x^* \in Q$  is a solution to \eqref{prob:VIstrong} if and only if $x^* \in Q$ is a solution to \eqref{min_problem}.
\end{example}

\begin{example}[Saddle point problem]\label{ex:saddleproblem}  
 Let us consider the saddle point problem
\begin{equation}\label{minmax_problem}
    \min_{u \in Q_u}\max_{v \in Q_v}  f(u, v), 
\end{equation}
\end{example}
and assume that $F(x) : = F(u, v) = \left(\nabla_u f(u,v), -\nabla_v f(u, v)\right)^{\top}$, where $Q = Q_u \times Q_v$ with $Q_u \subseteq \mathbb{R}^{n_u}, Q_v \subseteq \mathbb{R}^{n_v}$. Then if $f$ is convex in $u$ and concave in $v$, it can be proved that $x^* \in Q$ is a solution to \eqref{prob:VIstrong} if and only if $x^* = (u^*, v^*) \in Q$ is a solution to \eqref{minmax_problem}. 

\begin{example}[Fixed point problem]\label{ex:fixedproblem}
Let us consider the fixed point
problem
\begin{equation}\label{fixed_prob}
    \text{find} \;\; x^* \in Q \;\; \text{such that} \quad T(x^*) = x^*,
\end{equation}
where $T: \mathbb{R}^n \longrightarrow \mathbb{R}^n $ is an operator. By taking $F(x)  = x - T(x)$, it can be proved that $x^* \in Q = \mathbb{R}^n$ is a solution to \eqref{prob:VIstrong} if $F(x^*) = \textbf{0} \in \mathbb{R}^n$, i.e., $x^*$ is a solution to \eqref{fixed_prob}. 
\end{example}

\begin{definition}
For some $\varepsilon >0$, we call a point $\widehat{x} \in Q$ an $\varepsilon$-solution of the problem \eqref{main_constrained_prob}, if 
\begin{equation}
    \left\langle F(x), \widehat{x} - x \right \rangle \leq \varepsilon,  \quad  \forall x \in Q. 
\end{equation}
\end{definition}

Following \cite{nesterov2007dual}, to assess the quality of a candidate solution $\widehat{x}$, we use the following restricted  gap (or merit) function
\begin{equation}\label{eq:gap}
    \operatorname{Gap}(\widehat{x}) = \max_{u \in Q} \langle F(u), \widehat{x}  - u \rangle .
\end{equation}

Thus, our goal is to find an approximate solution to the problem \eqref{main_constrained_prob}, that is, a point $\widehat{x} \in Q$ such that the following inequality holds
\begin{equation}\label{eq:Appr}
    \operatorname{Gap}(\widehat{x}) = \max_{u \in Q} \langle F(u), \widehat{x}  - u \rangle \leq \varepsilon,
\end{equation}
for some $\varepsilon >0$.

For problem \eqref{main_constrained_prob}, we propose an Algorithm \ref{alg_mirror_descent}, under consideration
\begin{equation}\label{cond_x1xstar}
    V_{\psi} (x, y) \leq  R < \infty   
    \quad \forall x,y \in Q,  
\end{equation}
for some $R> 0$. 

\begin{algorithm}[!ht]
\caption{Mirror descent method for constrained variational inequality problem.}\label{alg_mirror_descent}
\begin{algorithmic}[1]
\REQUIRE step sizes $\{\gamma_k\}_{k \geq 1}$,  initial point $x^1  \in Q$, number of iterations $N$.
\FOR{$k= 1, 2, \ldots, N$}
\STATE $x^{k+1} = \arg\min\limits_{x \in Q} \left\{ \langle x, F(x^k)  \rangle + \frac{1}{\gamma_k} V_{\psi}(x,x^k) \right\} $.
\ENDFOR
\end{algorithmic}
\end{algorithm}

For Algorithm \ref{alg_mirror_descent}, we have the following result.

\begin{theorem}\label{theo_main_ineq_mirror_desc}
Let $F: Q \longrightarrow \textbf{E}^*$ be a continuous, bounded, and $\delta$-monotone operator. Then for problem \eqref{main_constrained_prob}, by Algorithm \ref{alg_mirror_descent}, with a positive non-increasing sequence of step sizes $\{\gamma_k\}_{k \geq 1}$, for any fixed $m \geq -1$, it satisfies the following inequality
\begin{equation}\label{main_ineq_mirror_desc}
    \operatorname{Gap} (\widehat{x})  \leq  \frac{1}{\sum_{k= 1}^{N} \gamma_k^{-m}} \left( \frac{R}{\gamma_N^{m+1}} + \frac{1}{2 \sigma} \sum_{k= 1}^{N} \frac{ \|F(x^k)\|_*^2}{\gamma_k^{m -1}} \right) + \delta,
\end{equation}
where 
$$
    \widehat{x} = \frac{1}{\sum_{k= 1}^{N} \gamma_k^{-m}} \sum_{k = 1}^{N} \gamma_k^{-m} x^k .
$$ 
\end{theorem}

\begin{proof}
Let $\widetilde{f}(x) :=  \left\langle x, F(x^k)  \right\rangle + \frac{1}{\gamma_k} V_{\psi}(x,x^k)$. From Algorithm \ref{alg_mirror_descent}, we have $x^{k+1} =\arg\min\limits_{x \in Q} \widetilde{f}(x) $. By the optimality condition, we get
$$
    \left\langle \nabla \widetilde{f}(x^{k+1}), x - x^{k+1} \right\rangle \geq 0, \quad \forall x \in Q.
$$
Thus, 
$$
    \left\langle \gamma_k F(x^k) + \nabla \psi(x^{k+1})- \nabla \psi(x^k), x - x^{k+1} \right\rangle \geq 0, \quad \forall x \in Q.
$$
i.e., 
\begin{equation*}
    \left\langle \gamma_k F(x^k), x^{k+1} - x  \right\rangle \leq - \left\langle \nabla \psi(x^{k}) -  \nabla \psi(x^{k+1}), x - x^{k+1}\right\rangle, \quad \forall x \in Q.
\end{equation*}

By Lemma \ref{three_points_lemma}, for any $x \in Q$, we have 
\begin{align*}
     & \quad - \left\langle \nabla \psi(x^{k}) -  \nabla \psi(x^{k+1}), x - x^{k+1}\right\rangle,
    \\& = - \left( V_{\psi} (x, x^{k+1}) + V_{\psi} (x^{k+1}, x^k) - V_{\psi} (x, x^{k}) \right) 
    \\& = V_{\psi} (x, x^{k}) - V_{\psi} (x, x^{k+1}) - V_{\psi} (x^{k+1}, x^k). 
\end{align*}
Thus, we get
\[
    \left\langle \gamma_k F(x^k), x^{k+1} - x  \right\rangle \leq V_{\psi} (x, x^{k}) - V_{\psi} (x, x^{k+1}) - V_{\psi} (x^{k+1}, x^k), \quad \forall x \in Q.
\]
This means that for any $x \in Q$, we have
\begin{align*}
    \gamma_k \left\langle  F(x^k), x^{k} - x  \right\rangle & \leq V_{\psi} (x, x^{k}) - V_{\psi} (x, x^{k+1}) - V_{\psi} (x^{k+1}, x^k)
    \\& \quad + \left\langle \gamma_k F(x^k), x^k - x^{k+1} \right\rangle. 
\end{align*}

By the Fenchel-Young inequality \eqref{Fenchel_Young_ineq}, with $\lambda = \sigma > 0 $, we find 
$$
    \left\langle \gamma_k F(x^k), x^k - x^{k+1} \right\rangle \leq \frac{\gamma_k^2}{2 \sigma} \|F(x^k)\|_*^2 + \frac{\sigma}{2} \|x^k - x^{k+1}\|^2. 
$$
Therefore, for any $x \in Q$, we get
\begin{align*}
    \gamma_k \left \langle F(x^k), x^k - x \right\rangle & \leq V_{\psi} (x, x^{k}) - V_{\psi} (x, x^{k+1}) - V_{\psi} (x^{k+1}, x^k) 
    \\& \quad + \frac{\gamma_k^2}{2 \sigma} \|F(x^k)\|_*^2  + \frac{\sigma}{2} \|x^k - x^{k+1}\|^2. 
\end{align*}

But from \eqref{eq_breg}, we have 
$$
    V_{\psi} (x^{k+1}, x^{k}) \geq \frac{\sigma}{2} \|x^{k+1} - x^k\|^2 .
$$
Thus, for any $x \in Q$, we get the following inequality
\begin{equation}\label{eq:452df}
    \left \langle F(x^k), x^k - x \right\rangle \leq \frac{1}{\gamma_k} \left( V_{\psi} (x, x^{k}) - V_{\psi} (x, x^{k+1}) \right) + \frac{\gamma_k}{2 \sigma} \| F(x^k)\|_*^2. 
\end{equation}

Since $F$ is $\delta$-monotone operator, we get 
\[
    \left\langle F(x), x^k - x\right\rangle - \delta \leq \left\langle F(x^k), x^k - x \right\rangle, \quad \forall x \in Q. 
\]
Therefore, for any $x \in Q$, we have
\begin{equation}\label{eq_2}
    \left \langle F(x), x^k - x \right\rangle \leq \frac{1}{\gamma_k} \left( V_{\psi} (x, x^{k}) - V_{\psi} (x, x^{k+1}) \right) + \frac{\gamma_k}{2 \sigma} \| F(x^k)\|_*^2 + \delta. 
\end{equation}

By multiplying both sides of \eqref{eq_2} by $\frac{1}{\gamma_k^m}$, and taking the sum from $1$ to $N$, for any $x \in Q$, we get
\begin{align}
    \sum_{k = 1}^{N} \frac{1}{\gamma_k^m} \left \langle F(x), x^k - x \right\rangle & \leq  \sum_{k = 1}^{N}\frac{1}{\gamma_k^{m+1}} \left( V_{\psi} (x, x^{k}) - V_{\psi} (x, x^{k+1}) \right)  \nonumber
    \\& \;\;\;\; + \frac{1}{2 \sigma}\sum_{k = 1}^{N} \frac{\|F(x^k)\|_*^2}{ \gamma_k^{m-1}} + \sum_{k = 1}^{N} \frac{\delta}{\gamma_k^{m}}. \label{ghjklk}
\end{align}
But, 
\begin{align}\label{eq:gfdfd}
     \sum_{k = 1}^{N} \frac{1}{\gamma_k^m} \left \langle F(x), x^k - x \right\rangle   & = \left( \sum_{k = 1}^{N}\gamma_k^{-m}\right) \left \langle F(x), \frac{1}{ \sum_{k = 1}^{N} \gamma_k^{-m}} \sum_{k = 1}^{N} \gamma_k^{-m} x^k - x \right\rangle \nonumber
     \\& = \left( \sum_{k = 1}^{N}\gamma_k^{-m}\right) \left \langle F(x), \widehat{x} - x \right\rangle. 
\end{align}

For any $x \in Q$, we have
\begin{align*}
    & \quad \; \sum_{k = 1}^{N}\frac{1}{\gamma_k^{m+1}} \left( V_{\psi} (x, x^{k}) - V_{\psi} (x, x^{k+1}) \right) \nonumber
    \\& = \frac{1}{\gamma_1^{m+1}} \left( V_{\psi}(x, x^1)  - V_{\psi}(x, x^2)\right)    + \sum_{k = 2}^{N-1}\frac{1}{\gamma_k^{m+1}} \left( V_{\psi}(x, x^k)  - V_{\psi}(x, x^{k+1})\right) \nonumber
    \\& \;\;\;\; + \frac{1}{\gamma_N^{m+1}} \left( V_{\psi}(x, x^N)  - V_{\psi}(x, x^{N+1})\right)  \nonumber
    \\& \leq \frac{1}{\gamma_1^{m+1}} V_{\psi}(x, x^1) +   \frac{1}{\gamma_N^{m+1}} V_{\psi}(x, x^N) + \sum_{k = 2}^{N-1} \frac{1}{\gamma_k^{m+1}} V_{\psi}(x, x^k)  \nonumber
    \\& \quad  - \frac{1}{\gamma_1^{m+1}} V_{\psi}(x, x^2)   - \sum_{k = 2}^{N-1} \frac{1}{\gamma_k^{m+1}} V_{\psi}(x, x^{k+1})  \nonumber
    \\& = \frac{1}{\gamma_1^{m+1}} V_{\psi}(x, x^1) + \sum_{k = 2}^{N} \frac{1}{\gamma_k^{m+1}} V_{\psi}(x, x^k)  - \sum_{k = 2}^{N} \frac{1}{\gamma_{k-1}^{m+1}} V_{\psi}(x, x^k) \nonumber
    \\& = \frac{1}{\gamma_1^{m+1}} V_{\psi}(x, x^1) + \sum_{k = 2}^{N} \left( \frac{1}{\gamma_k^{m+1}} -\frac{1}{\gamma_{k-1}^{m+1}} \right)  V_{\psi}(x, x^k) \nonumber
   \\& \leq \frac{R}{\gamma_1^{m+1}}  + R \sum_{k = 2}^{N} \left( \frac{1}{\gamma_k^{m+1}} - \frac{1}{\gamma_{k-1}^{m+1}} \right)  \nonumber
   \\& = \frac{R}{\gamma_1^{m+1}} +  R\left(-\frac{1}{\gamma_1^{m+1}} + \frac{1}{\gamma_{N}^{m+1}} \right) = \frac{R}{\gamma_{N}^{m+1}}. 
\end{align*}

Therefore, from \eqref{ghjklk}, \eqref{eq:gfdfd}, and the last inequality, we get
\begin{equation*}
    \left(\sum_{k = 1}^{N} \gamma_k^{-m}\right) \max_{x \in Q} \left \langle F(x), \widehat{x} - x \right\rangle  \leq  \frac{R}{\gamma_N^{m+1}}  + \frac{1}{2 \sigma} \sum_{k = 1}^{N} \frac{\|F(x^k)\|_*^2}{\gamma_k^{m-1}} +   \sum_{k = 1}^{N} \frac{\delta}{\gamma_k^{m}}. 
\end{equation*}

By dividing by $\sum_{k= 1}^{N} \gamma_k^{-m}$,  we get the desired inequality
\begin{equation*}
    \max_{x \in Q} \left \langle F(x), \widehat{x} - x \right\rangle  \leq  \frac{1}{\sum_{k= 1}^{N} \gamma_k^{-m}} \left( \frac{R}{\gamma_N^{m+1}} + \frac{1}{2 \sigma} \sum_{k= 1}^{N} \frac{\|F(x^k)\|_*^2}{\gamma_k^{m -1}} \right) + \delta.
\end{equation*}
\end{proof}

Now, let us see the convergence rate of Algorithm \ref{alg_mirror_descent} with the following (non-adaptive) time-varying step size rule
\begin{equation}\label{steps_rules}
    \gamma_k = \frac{\sqrt{2 \sigma}}{L_F \sqrt{k}}, \quad   k = 1, 2, \ldots, N, 
\end{equation}
and different values of the parameter $m \geq -1$. 

We mention here that for the following adaptive step size scheme
\begin{equation}\label{adaptive_steps}
    \gamma_k = \frac{\sqrt{2 \sigma}}{\|F(x^k)\|_* \sqrt{k}} \quad k = 1, \ldots N, 
\end{equation}
there is no guarantee that the sequence $\{\gamma_k\}_{k \geq 1}$ is non-increasing. As a result, this step size does not satisfy one of the conditions of Theorem \ref{theo_main_ineq_mirror_desc}. However, Algorithm \ref{alg_mirror_descent} performs well in practice, and the convergence rate remains the same as the step size defined in \eqref{steps_rules}.  Therefore, we need to consider restructuring the scheme of step sizes so that it becomes adaptive and allows us to obtain the same results as previously obtained for the non-adaptive ones.

\begin{corollary}\label{corollary_mirror_m_minus1}
Let $F: Q \longrightarrow \textbf{E}^*$ be a continuous, bounded, and $\delta$-monotone operator. Then for problem \eqref{main_constrained_prob}, by Algorithm \ref{alg_mirror_descent}, with  $m = -1$, and the time-varying step size given in \eqref{steps_rules}, it satisfies the following 
\begin{equation}\label{rate_mirror_k_minus1}
    \operatorname{Gap} (\widetilde{x}) \leq \frac{L_F \left(1+R + \log(N)  \right)}{\sqrt{\sigma }} \cdot \frac{1}{\sqrt{N}} + \delta  = O\left(\frac{\log(N)}{\sqrt{N}}\right) + \delta, 
\end{equation}
where $\widetilde{x} = \frac{1}{\sum_{k = 1}^{N} \gamma_k} \sum_{k= 1}^{N} \gamma_k x^k$. 
\end{corollary}
\begin{proof}
By setting $m = -1$ in \eqref{main_ineq_mirror_desc}, we get the following inequality
\begin{equation}\label{eq_3}
    \operatorname{Gap} (\widetilde{x}) \leq \frac{1}{\sum_{k = 1}^{N} \gamma_k}\left(R + \frac{1}{2 \sigma} \sum_{k = 1}^{N} \gamma_k^2 \|F(x^k)\|_*^2\right) + \delta, 
\end{equation}
where $\widetilde{x} = \frac{1}{\sum_{k = 1}^{N} \gamma_k} \sum_{k= 1}^{N} \gamma_k x^k$. 

When $\gamma_k = \frac{\sqrt{2 \sigma}}{L_F \sqrt{k}}, \, k = 1, 2, \ldots, N$, and since $\|F(x^k)\|_* \leq L_F$, then by substitution in \eqref{eq_3} we find
$$
    \operatorname{Gap} (\widetilde{x}) = \max_{x \in Q} \left \langle F(x), \widetilde{x} - x \right\rangle \leq \frac{L_F}{\sqrt{2 \sigma}} \cdot \frac{R + \sum_{k = 1}^{N} \frac{1}{k} }{\sum_{k = 1}^{N} \frac{1}{\sqrt{k}}} + \delta,
$$
where $\widetilde{x} = \frac{1}{\sum_{k = 1}^{N} \frac{1}{\sqrt{k}}} \sum_{k= 1}^{N} \frac{1}{\sqrt{k}} x^k$. But
\begin{equation*}\label{eq_4}
    \sum_{k = 1}^{N} \frac{1}{k} \leq 1 + \log(N), \quad \text{and} \quad  \sum_{k = 1}^{N} \frac{1}{\sqrt{k}} \geq 2 \sqrt{N+1} - 2, \quad \forall N \geq 1.
\end{equation*}
Therefore,
\begin{equation*}
    \operatorname{Gap} (\widetilde{x})  \leq  \frac{L_F}{\sqrt{2 \sigma}} \cdot \frac{1+R  + \log(N)}{2 \sqrt{N+1} - 2} + \delta  \leq \frac{L_F}{\sqrt{ \sigma}} \cdot \frac{ 1 + R + \log(N)}{\sqrt{N}} + \delta.  
\end{equation*}
In the last inequality, we used the fact $2 \sqrt{2} \left( \sqrt{N+1} - 1\right) \geq \sqrt{N}, \, \forall N \geq 1$.  
\end{proof}

Note that the convergence rate in \eqref{rate_mirror_k_minus1} is suboptimal for the bounded monotone operators (i.e., when $\delta = 0$ in \eqref{d_monot}). 

\bigskip
From Theorem \ref{theo_main_ineq_mirror_desc}, with a special value of the parameter $m$, we can obtain the optimal convergence rate $O\left(\frac{1}{\sqrt{N}}\right)$ of Algorithm \ref{alg_mirror_descent}, with the time-varying step size given in \eqref{steps_rules} and monotone operators (i.e., $\delta = 0$ in \eqref{d_monot}). 

For this, we have the following result. 

\begin{corollary}\label{corollary_mirror_m_0}
Let $F: Q \longrightarrow \textbf{E}^*$ be a continuous, bounded, and $\delta$-monotone operator.  Then for problem \eqref{main_constrained_prob}, by Algorithm \ref{alg_mirror_descent}, with  $m = 0$, and the time-varying step size given in \eqref{steps_rules}, it satisfies the following
\begin{equation}\label{rate_mirror_k_0}
    \operatorname{Gap}(\overline{x}) \leq \frac{L_F \left( 2+ R \right)}{\sqrt{2\sigma}} \cdot \frac{1}{\sqrt{N}} + \delta = O\left(\frac{1}{\sqrt{N}}\right) + \delta, 
\end{equation}
where $\overline{x} = \frac{1}{N} \sum_{k= 1}^{N} x^k$. 
\end{corollary}
\begin{proof}
By setting $m = 0$ in \eqref{main_ineq_mirror_desc}, we get 
\begin{equation}\label{eq_5}
    \operatorname{Gap}(\overline{x}) \leq \frac{1}{N} \left( \frac{R}{\gamma_N} + \frac{1}{2 \sigma} \sum_{k = 1}^{N} \gamma_k \|F(x^k)\|_*^2  \right) + \delta, 
\end{equation}
where $\overline{x}= \frac{1}{N} \sum_{k= 1}^{N} x^k$. 

When $\gamma_k = \frac{\sqrt{2 \sigma}}{L_F \sqrt{k}}, \, k = 1, 2, \ldots, N$, and since $\|F(x^k)\|_* \leq L_F$,  then by substitution in \eqref{eq_5} we find
$$
    \operatorname{Gap}(\overline{x}) \leq  \frac{1}{N} \left(  \frac{R L_F \sqrt{N} }{\sqrt{2 \sigma}} + \frac{L_F}{\sqrt{2 \sigma}} \sum_{k = 1}^{N} \frac{1}{\sqrt{k}} \right) + \delta.
$$
But
\begin{equation*}\label{eq_6}
\sum_{k = 1}^{N} \frac{1}{\sqrt{k}} \leq 2 \sqrt{N}, \quad \forall N \geq 1.
\end{equation*}
Therefore,
\begin{equation*}
    \operatorname{Gap}(\overline{x}) \leq  \frac{1}{N} \cdot \frac{L_F}{\sqrt{2 \sigma}} \left( R \sqrt{N}  + 2 \sqrt{N}\right) + \delta = \frac{L_F \left(2 + R\right)}{\sqrt{2 \sigma}}  \cdot \frac{1}{\sqrt{N}} + \delta.
\end{equation*}
\end{proof}

Also, the same optimal convergence rate for Algorithm \ref{alg_mirror_descent}, with bounded monotone operators (i.e., $\delta = 0$), can be obtained with any fixed $m \geq 1$, and time-varying step size given in \eqref{steps_rules}. 

For this, we have the following result.

\begin{corollary}\label{corollary_mirror_m_all}
Let $F: Q \longrightarrow \textbf{E}^*$ be a continuous, bounded, and $\delta$-monotone operator. Then for problem \eqref{main_constrained_prob}, by Algorithm \ref{alg_mirror_descent}, with any $m \geq 1$, and the time-varying step size given in \eqref{steps_rules}, it satisfies the following 
\begin{equation}\label{rate_mirror_m_all}
    \operatorname{Gap}(\widehat{x})  \leq \frac{L_F (m + 2) (1+ R)}{2 \sqrt{2 \sigma}} \cdot \frac{1}{\sqrt{N}} + \delta = O\left(\frac{1}{\sqrt{N}}\right) + \delta, 
\end{equation}
where $\widehat{x} = \frac{1}{\sum_{k= 1}^{N} \gamma_k^{-m}} \sum_{k = 1}^{N} \gamma_k^{-m} x^k $. 
\end{corollary}
\begin{proof} 
When $\gamma_k = \frac{\sqrt{2 \sigma}}{L_F \sqrt{k}}, \, k = 1, 2, \ldots, N$, and since $\|F(x^k)\|_* \leq L_F$, then by substitution in \eqref{main_ineq_mirror_desc} we find
$$
    \operatorname{Gap}(\widehat{x}) \leq \frac{L_F}{\sqrt{2 \sigma}} \cdot \frac{1}{\sum_{k = 1}^{N} \left(\sqrt{k}\right)^{m} }  \left(R \left(\sqrt{N}\right)^{m+1}  +  \sum_{k = 1}^{N} \left(\sqrt{k}\right)^{m-1} \right) + \delta.
$$

But, for any $m \geq 1$ and $N \geq 1$, 
\begin{equation*}
    \int_{0}^{N}\left(\sqrt{k}\right)^{m} dk \leq \sum_{k = 1}^{N}  \left(\sqrt{k}\right)^{m} \;\;\Longrightarrow\;\; \sum_{k = 1}^{N}  \left(\sqrt{k}\right)^{m} \geq \frac{2\left(\sqrt{N}\right)^{m+2}}{m+2},
\end{equation*}
and
$$
    \sum_{k = 1}^{N} \left(\sqrt{k}\right)^{m-1} \leq N \left(\sqrt{N}\right)^{m-1} = \left(\sqrt{N}\right)^{m+1}, \quad \forall m \geq 1, N\geq 1. 
$$

Therefore,
\begin{align*}
    \operatorname{Gap}(\widehat{x}) & \leq  \frac{L_F}{\sqrt{2 \sigma}} \cdot  \frac{m + 2}{2 \left(\sqrt{N}\right)^{m+2}}  \left(R \left(\sqrt{N}\right)^{m+1} + \left(\sqrt{N}\right)^{m+1}\right) + \delta
    \\& = \frac{L_F (m + 2) (1+ R)}{2 \sqrt{2 \sigma}} \cdot \frac{1}{\sqrt{N}} + \delta = O \left(\frac{1}{\sqrt{N}}\right) + \delta. 
\end{align*}
\end{proof}

\begin{remark}\label{remark_k_gretear_minus1_mirror}
In the comparison with suboptimal convergence rate \eqref{rate_mirror_k_minus1}, when $m \geq 1$, the weighting scheme $\frac{1}{\sum_{k= 1}^{N} \gamma_k^{-m}} \sum_{k = 1}^{N} \gamma_k^{-m} x^k$ assigns smaller weights to the initial points and larger weights to the most recent points that generated by Algorithm \ref{alg_mirror_descent}. This fact will be shown in numerical experiments (see Sect. \ref{sect_numerical}). 
\end{remark}

\section{Mirror-Descent Method for Variational Inequality Problem with Functional Constraints}\label{sectinAlgs}

Consider a set of convex subdifferentiable functionals $g_i: Q \longrightarrow\mathbb{R}$,  $i = 1,2, \ldots, p$. Also we assume that all functionals $ g_i $ are Lipschitz-continuous with some constant $ M_{g_i} >0 $, i.e.,
\begin{equation}\label{Lip_condition_for_gm}
\left|g_i(x)-g_i(y) \right|\leq M_{g_i} \|x-y\|, \quad \forall \; x,y \in Q \;\; \text{and} \;\; i =1, \ldots, p.
\end{equation}
This means that at every point $x \in Q$ and for any $ i = 1, \ldots, p$ there is a subgradient $\nabla g_i(x)$, such that $\|\nabla g_i(x)\|_{*} \leq M_{g_i}$. 

In this section, we consider the following variational inequality problem
\begin{equation}\label{main_func_cons_1}
\begin{aligned}
    \text{Find} \quad x^* \in Q & : \quad  \langle F(x), x^* - x \rangle \leq 0 \quad \forall x \in Q, 
    \\& \text{and} \quad g_i(x) \leq 0, \quad \forall i = 1, 2 \ldots, p, 
\end{aligned}   
\end{equation}
where $F: Q \longrightarrow \textbf{E}^*$ is a continuous, bounded (i.e., \eqref{bound_cond} holds), and $\delta$-monotone operator (i.e., \eqref{d_monot} holds). 

It is clear that instead of a set of Lipschitz-continuous functions  $\{g_i(\cdot)\}_{i=1}^{p}$ we can see one Lipschitz-continuous functional constraint $g: Q \longrightarrow \mathbb{R}$, such that
\begin{equation}\label{functional_of_constraints_g(x)}
    g(x) = \max\limits_{1 \leq i \leq p} \{g_i(x)\}, \;\; \text{and} \;\; |g(x)-g(y)|\leq M_g\|x-y\| \; \quad \forall \; x,y\in Q,
\end{equation}
where $M_g = \max_{1 \leq i \leq p} \{M_{g_i}\}$. Thus, the problem \eqref{main_func_cons_1}, will be equivalent to the following problem
\begin{equation}\label{main_func_cons_2}
    \text{Find} \quad x^* \in Q  : \quad  \langle F(x), x^* - x \rangle \leq 0,  \;\; \text{and} \;\; g(x) \leq 0, \quad \forall x \in Q.  
\end{equation}

\begin{definition}
For some $\varepsilon >0$, we call a point $\widehat{x} \in Q$ an $\varepsilon$-solution of the problem \eqref{main_func_cons_2}, if 
\begin{equation}
    \left\langle F(x), \widehat{x} - x \right \rangle \; \leq \varepsilon \quad \forall x \in Q, \quad \text{and} \quad g(\widehat{x}) \leq \varepsilon. 
\end{equation}
\end{definition}

To solve the problem \eqref{main_func_cons_1} (or its equivalent \eqref{main_func_cons_2}), we propose a mirror-decent type method, listed as Algorithm \ref{alg:MD_func_cons} below. 

As can be seen from the items of Algorithm \ref{alg:MD_func_cons}, the needed point (i.e., the output, see \eqref{eq:output_alg2}) is selected among the points $x^i$ for which $g(x^i) \leq \varepsilon$. Therefore, we will call step $i$ \textit{productive} if $g(x^i) \leq \varepsilon$. If the reverse inequality $g(x^i) > \varepsilon$ holds, then the step $i$ will be called \textit{non-productive}.

Let $I$ and $J$ denote the set of indices of productive and non-productive steps, respectively. $|A|$ denotes the cardinality of the set $A$.  Let us also set $\gamma_k := \gamma_k^F$ if $k \in I$, $\gamma_k := \gamma_k^g$ if $k \in J$.

\begin{algorithm}[H]
\caption{Mirror descent algorithm for VIs with functional constraints.}\label{alg:MD_func_cons}
\begin{algorithmic}[1]
\REQUIRE $\varepsilon>0$, initial point $x^1  \in Q$, step sizes $\{\gamma_k^F\}_{k \geq 1}, \{\gamma_k^g\}_{k \geq 1}$, number of iterations $N$.
\STATE $I \longrightarrow \emptyset, J \longrightarrow \emptyset. $
\FOR{$k= 1, 2, \ldots, N$}
\IF{$g(x^k) \leq \varepsilon$}
\STATE $x^{k+1} = \arg\min_{x \in Q} \left\{ \langle x, F(x^k)  \rangle + \frac{1}{\gamma_k^F} V_{\psi}(x,x^k) \right\} $.
\STATE  $k \longrightarrow I$  \qquad "productive step"
\ELSE
\STATE Calculate $\nabla g(x^k) \in \partial g(x^k)$,
\STATE $x^{k+1} = \arg\min_{x \in Q} \left\{ \langle x, \nabla g(x^k)  \rangle + \frac{1}{\gamma_k^g} V_{\psi}(x,x^k) \right\} $.
\STATE  $k \longrightarrow J$  \qquad "non-productive step"
\ENDIF
\ENDFOR
\end{algorithmic}
\end{algorithm}

For Algorithm \ref{alg:MD_func_cons}, we have the following result. 

\begin{theorem}\label{theo:alg_func_cons}
Let $F: Q \longrightarrow \textbf{E}^*$ be a continuous, bounded, and $\delta$-monotone operator. Let  $g(x) = \max_{1 \leq i \leq p} \{g_i(x)\}$ be an $M_g$-Lipschitz convex function, where $g_i: Q \longrightarrow \mathbb{R},\; \forall i = 1,2, \ldots, p$ are $M_{g_i}$-Lipschitz convex functions, and $M_g = \max_{1 \leq i \leq p} \{M_{g_i}\}$. Then for problem \eqref{main_func_cons_2}, by Algorithm \ref{alg:MD_func_cons}, with a positive non-increasing sequence of step sizes $\{\gamma_k\}_{k \geq 1}$, for any fixed $m \geq -1$,  after $N \geq 1$ iterations, it satisfies the following inequality
\begin{align}\label{main_ineq_MD_func_cons}
    \operatorname{Gap} (\widehat{x}) < \frac{1}{\sum_{k \in I} (\gamma_k^F)^{-m}} & \Bigg( \frac{R}{\gamma_N^{m+1}} + \frac{1}{2 \sigma} \sum_{k\in I} \frac{\|F(x^k)\|_*^2}{(\gamma_k^F)^{m -1}} + \frac{1}{2 \sigma} \sum_{k\in J} \frac{\|\nabla g(x^k)\|_*^2}{(\gamma_k^g)^{m -1}} \nonumber
    \\& \qquad \qquad \qquad \qquad  - (\varepsilon - M_g D)\sum_{k\in J} (\gamma_k^g)^{-m}  \Bigg) + \delta,
\end{align}
where $D > 0$ is the diameter of $Q$, and 
\begin{equation}\label{eq:output_alg2}
    \widehat{x} = \frac{1}{\sum_{k \in I} (\gamma_k^F)^{-m}} \sum_{k \in I}  (\gamma_k^F)^{-m}x^k.
\end{equation}
\end{theorem}
\begin{proof}
Similar to what was done in the proof of Theorem \ref{theo_main_ineq_mirror_desc}, we find that for any $k \in I$ and $x \in Q$ (see \eqref{eq:452df}),  
\[
    \left \langle F(x^k), x^k - x \right\rangle \leq \frac{1}{\gamma_k^F} \left( V_{\psi} (x, x^{k}) - V_{\psi} (x, x^{k+1}) \right) + \frac{\gamma_k^F}{2 \sigma} \| F(x^k)\|_*^2.
\]

By multiplying both sides of the previous inequality with $\frac{1}{(\gamma_k^F)^m}$, and since $F$ is $\delta$-monotone, i.e., 
$$
    \left \langle F(x^k), x^k - x \right\rangle \geq 
    \left \langle F(x), x^k - x \right\rangle - \delta, \quad \forall x \in Q, 
$$
we get (for any $k \in I$ and $x \in Q$)
\begin{align}\label{ineq_productive}
    \frac{1}{(\gamma_k^F)^m}  \left \langle F(x), x^k - x \right\rangle & \leq  \frac{1}{(\gamma_k^F)^{m+1}} \left( V_{\psi} (x, x^{k}) - V_{\psi} (x, x^{k+1}) \right) \nonumber 
    \\& \quad +  \frac{\|F(x^k)\|_*^2}{2 \sigma (\gamma_k^F)^{m-1}}  +  \frac{\delta}{(\gamma_k^F)^{m}}.
\end{align}

Also, for any $k \in J$ and $x \in Q$, we have 
\begin{equation}\label{ineq_nonproductive}
    \frac{g(x^k) - g(x)}{(\gamma_k^g)^m} \leq \frac{1}{(\gamma_k^g)^{m+1}} \left(V_{\psi} (x, x^{k}) - V_{\psi} (x, x^{k+1})\right) + \frac{\| \nabla g(x^k)\|_*^2}{2 \sigma (\gamma_k^g)^{m-1}}. 
\end{equation}

By taking the summation, in each side of \eqref{ineq_productive} and \eqref{ineq_nonproductive}, over productive and non-productive steps, with $\gamma_k = \gamma_k^F$ if $k \in I$ and $\gamma_k = \gamma_k^g$ if $k \in J$, we get the following (for any $x \in Q$) 
\begin{align}\label{eqq0}
    & \quad \; \sum_{k \in I} (\gamma_k^F)^{-m} \langle F(x), x^k - x\rangle + \sum_{k \in J} (\gamma_k^g)^{-m} (g(x^k) - g(x)) \nonumber
    \\& \leq \sum_{k = 1}^{N} \frac{1}{\gamma_k^{m+1}} \left( (V_{\psi} (x, x^{k}) - V_{\psi} (x, x^{k+1}) \right) + \frac{1}{2 \sigma} \sum_{k \in I} \frac{\|F(x^k)\|_*^2}{ (\gamma_k^F)^{m-1}} \nonumber
    \\& \quad + \frac{1}{2 \sigma} \sum_{k \in J} \frac{\|\nabla g(x^k)\|_*^2}{ (\gamma_k^g)^{m-1}}  + \sum_{k \in I} \frac{\delta}{(\gamma_k^F)^{m}}. 
\end{align}
But for any $x \in Q$, we have
\begin{align*}
    & \quad \sum_{k \in I} (\gamma_k^F)^{-m} \left\langle F(x), x^k - x\right\rangle  = \left\langle F(x), \sum_{k \in I} (\gamma_k^F)^{-m} x^k - \sum_{k \in I} (\gamma_k^F)^{-m} x  \right \rangle  \nonumber
    \\& = \left(\sum_{k \in I} (\gamma_k^F)^{-m}\right) \left\langle F(x), \frac{1}{\sum_{k \in I} (\gamma_k^F)^{-m}} \sum_{k \in I} (\gamma_k^F)^{-m} x^k - x \right \rangle \nonumber
    \\& = \left(\sum_{k \in I} (\gamma_k^F)^{-m}\right) \left\langle F(x), \widehat{x} - x \right \rangle.
\end{align*}
Thus,  from the last inequality and \eqref{eqq0}, for any $x \in Q$,  we get 
\begin{align}\label{tredf}
    \left(\sum_{k \in I} (\gamma_k^F)^{-m}\right) \left\langle F(x), \widehat{x} - x \right \rangle & \leq \sum_{k = 1}^{N} \frac{1}{\gamma_k^{m+1}} \left( (V_{\psi} (x, x^{k}) - V_{\psi} (x, x^{k+1}) \right) \nonumber
    \\& \quad + \frac{1}{2 \sigma} \sum_{k \in I} \frac{\|F(x^k)\|_*^2}{ (\gamma_k^F)^{m-1}}  + \frac{1}{2 \sigma} \sum_{k \in J} \frac{\|\nabla g(x^k)\|_*^2}{ (\gamma_k^g)^{m-1}}  \nonumber
    \\& \quad + \delta \sum_{k \in I} (\gamma_k^F)^{-m} -\sum_{k \in J} (\gamma_k^g)^{-m} (g(x^k) - g(x)). 
\end{align}

Since, for any $k \in J$, we have \begin{equation}\label{eq_nonprod00}
g(x^k) - g(x^*) \geq g(x^k) > \varepsilon > 0.
\end{equation}
Then by the convexity of the function $g$, for any $x \in Q$, we have 
\begin{align}\label{fds12}
    & \quad\, - \sum_{k \in J} (\gamma_k^g)^{-m} (g(x^k) - g(x)) \nonumber
    \\& = - \sum_{k \in J} (\gamma_k^g)^{-m} (g(x^k) - g(x^*)) +  \sum_{k \in J} (\gamma_k^g)^{-m} (g(x) - g(x^*)) \nonumber
    \\& < - \varepsilon \sum_{k \in J} (\gamma_k^g)^{-m} + \sum_{k \in J} (\gamma_k^g)^{-m} \left\langle \nabla g(x), x - x^* \right\rangle \nonumber
    \\& \leq - \varepsilon \sum_{k \in J} (\gamma_k^g)^{-m} + M_g D  \sum_{k \in J} (\gamma_k^g)^{-m}, 
\end{align}
where in the last inequality, we used the Cauchy-Schwartz inequality and the fact that $\|\nabla g(x)\|_* \leq M_g, \, \forall x \in Q$, and $Q$ is bounded with a diameter $D > 0$.  

For any $x \in Q$, we have 
\begin{equation}\label{ds5a4d5s}
    \sum_{k = 1}^{N}\frac{1}{\gamma_k^{m+1}} \left( V_{\psi} (x, x^{k}) - V_{\psi} (x, x^{k+1}) \right) \leq \frac{R}{\gamma_{N}^{m+1}}.
\end{equation}

Therefore, by combining \eqref{fds12} and \eqref{ds5a4d5s} with \eqref{tredf}, for any $x \in Q$,  we get the following 
\begin{align*}
    \left(\sum_{k \in I} (\gamma_k^F)^{-m}\right) \left\langle F(x), \widehat{x} - x \right \rangle  
    & < \frac{R}{\gamma_{N}^{m+1}} + \frac{1}{2 \sigma} \sum_{k \in I} \frac{\|F(x^k)\|_*^2}{ (\gamma_k^F)^{m-1}} 
    \\& \quad + \frac{1}{2 \sigma} \sum_{k \in J} \frac{\|\nabla g(x^k)\|_*^2}{ (\gamma_k^g)^{m-1}}  + \delta \sum_{k \in I} (\gamma_k^F)^{-m}
    \\& \quad   - (\varepsilon - M_g D)  \sum_{k \in J} (\gamma_k^g)^{-m}.
\end{align*}

By dividing both sides of the last inequality by $\left(\sum_{k \in I} (\gamma_k^F)^{-m}\right)$, we get the desired inequality \eqref{main_ineq_MD_func_cons}.
\end{proof}

\begin{remark}[Stopping rule for Algorithm \ref{alg:MD_func_cons}]
From Theorem \ref{theo:alg_func_cons}, with 
$\widehat{x} = \frac{1}{\sum_{i\in I} (\gamma_i^F)^{-m}} \sum_{i \in I} (\gamma_i^F)^{-m} x^i$, for any $k \geq 1$, we find 
\begin{align*}
    & \quad \left(\sum_{i \in I} \left(\gamma_i^F\right)^{-m}\right)  \operatorname{Gap}(\widehat{x}) 
    \\& <  \frac{R}{\gamma_k^{m+1}} + \frac{1}{2 \sigma} \sum_{i\in I} \frac{\|F(x^i)\|_*^2}{(\gamma_i^F)^{m -1}} + \frac{1}{2 \sigma} \sum_{j\in J} \frac{\|\nabla g(x^j)\|_*^2}{(\gamma_j^g)^{m -1}}  - (\varepsilon - M_g D)\sum_{i = 1}^{k} (\gamma_i)^{-m}
    \\& \quad + (\varepsilon - M_g D)\sum_{i \in I } (\gamma_i^F)^{-m}+\delta \sum_{i \in I} \left(\gamma_i^F\right)^{-m}
    \\& = \left(\delta + \varepsilon\right) \sum_{i \in I} \left(\gamma_i^F\right)^{-m} - \Bigg( (\varepsilon - M_g D)\sum_{i = 1}^{k} (\gamma_i)^{-m} + M_gD\sum_{i \in I} \left(\gamma_i^F\right)^{-m}
    \\& \qquad \qquad \qquad \qquad \qquad  - \frac{R}{\gamma_k^{m+1}} - \frac{1}{2 \sigma} \sum_{i\in I} \frac{\|F(x^i)\|_*^2}{(\gamma_i^F)^{m -1}} - \frac{1}{2 \sigma} \sum_{j\in J} \frac{\|\nabla g(x^j)\|_*^2}{(\gamma_j^g)^{m -1}}\Bigg).
\end{align*}

From this, without relying on prior knowledge of the number of iterations $N$ that the algorithm performs, we can set for any $k \geq 1$,
\begin{align}\label{stop_criter1_alg2}
    M_gD\sum_{i \in I} \left(\gamma_i^F\right)^{-m}  & \geq \frac{R}{\gamma_k^{m+1}} + \frac{1}{2 \sigma} \sum_{i\in I} \frac{\|F(x^i)\|_*^2}{(\gamma_i^F)^{m -1}} + \frac{1}{2 \sigma} \sum_{j\in J} \frac{\|\nabla g(x^j)\|_*^2}{(\gamma_j^g)^{m -1}} \nonumber
    \\& \quad  +  ( M_g D - \varepsilon)\sum_{i = 1}^{k} (\gamma_i)^{-m}
\end{align}
as a stopping rule of Algorithm \ref{alg:MD_func_cons}. As a result, we conclude 
$$   
    \left(\sum_{i \in I} \left(\gamma_i^F\right)^{-m}\right)  \max_{x \in Q} \left\langle F(x), \widehat{x} - x \right \rangle <   (\delta + \varepsilon) \sum_{i \in I} \left(\gamma_i^F\right)^{-m}.
$$
Thus,  
\[
    \max_{x \in Q} \left\langle F(x), \widehat{x} - x \right \rangle <   \delta + \varepsilon.
\]

Note that for all $i \in I$ it holds that $g(x^i) \leq \varepsilon$, and since $g$ is convex, then we have 
$$
    g(\widehat{x}) \leq \frac{1}{\sum_{i \in I} \left(\gamma_i^F\right)^{-m}} \sum_{i \in I} \left(\gamma_i^f\right)^{-m} g(x^i) \leq \varepsilon.
$$

Thus after the stopping rule \eqref{stop_criter1_alg2} of Algorithm \ref{alg:MD_func_cons} is met we find $\widehat{x}$, which is given in \eqref{eq:output_alg2}, such that  
\[
    \operatorname{Gap}(\widehat{x}) = \max_{x \in Q} \left\langle F(x), \widehat{x} - x \right \rangle <   \delta + \varepsilon, \quad \text{and} \quad g(\widehat{x}) \leq \varepsilon.
\]
\end{remark}

\bigskip 

Now, let us take the following time-varying step size rule
\begin{equation}\label{steps_rulse_alg3}
    \gamma_k = 
    \begin{cases}
    \gamma_k^F : = \frac{\sqrt{2 \sigma}}{L_F \sqrt{k}}, \quad  \text{if} \; k \in I, \\
    \gamma_k^g : = \frac{\sqrt{2 \sigma}}{M_g \sqrt{k}}, \quad  \text{if} \; k \in J,
    \end{cases}
\end{equation}
and first, let us show for Algorithm \ref{alg:MD_func_cons}, with \eqref{steps_rulse_alg3}) that $|I| \ne 0$. For this, let us assume that $|I| = 0$; therefore, $|J| = N$,  i.e., all steps are non-productive.

Let $M: = L_F$ when we have a productive step and $M: = M_g$ when we have a non-productive step. From \eqref{eq_nonprod00} and \eqref{steps_rulse_alg3} we get
\begin{equation}\label{refdg12}
    \sum_{k = 1}^{N} \frac{g(x^k) - g(x^*)}{\gamma_{k}^m} > \sum_{k = 1}^{N} \frac{\varepsilon}{\gamma_{k}^m} = \frac{\varepsilon M^m}{\left(\sqrt{2 \sigma}\right)^{m}} \sum_{k = 1}^{N} \left(\sqrt{k}\right)^m, 
\end{equation}
and for all $k \in J = \{1, \ldots, N\}$,  we get
\begin{align*}
    \sum_{k = 1}^{N} \frac{g(x^k) - g(x^*)}{\gamma_{k}^m}  & \leq \frac{R}{\gamma_N^{m+1}} + \frac{1}{2 \sigma} \sum_{k = 1}^{N} \frac{\|\nabla g(x^k)\|_*^2}{\gamma_k^{m-1}}
    \\& \leq \frac{M^{m+1}}{\left(\sqrt{2 \sigma}\right)^{m+1}} \left(R \left(\sqrt{N}\right)^{m+1} + \sum_{k = 1}^{N} \left(\sqrt{k}\right)^{m-1} \right).
\end{align*}

But, it can be verified (numerically) that for a sufficiently big number of iterations $N$ (dependently on suitable values of the parameters $R > 0, m \geq -1, M> 0, \varepsilon > 0, \sigma > 0 $),  the following inequality holds
\begin{equation}\label{eq_nnnn}
    \frac{M^{m+1}}{\left(\sqrt{2 \sigma}\right)^{m+1}} \left(R \left(\sqrt{N}\right)^{m+1} + \sum_{k = 1}^{N} \left(\sqrt{k}\right)^{m-1} \right) < \frac{\varepsilon M^m}{\left(\sqrt{2 \sigma}\right)^{m}} \sum_{k = 1}^{N} \left(\sqrt{k}\right)^m. 
\end{equation}

Therefore, for a sufficiently big number $N \gg 1$,  we get 
$$
    \sum_{k = 1}^{N} \frac{g(x^k) - g(x^*)}{\gamma_{k}^m}  < \frac{\varepsilon M^m}{\left(\sqrt{2 \sigma}\right)^{m}} \sum_{k = 1}^{N} \left(\sqrt{k}\right)^m.
$$

So, we have a contradiction with \eqref{refdg12}. This means that $|I| \ne 0$. 

\begin{remark}
Note that the reverse inequality of \eqref{eq_nnnn}, i.e., 
$$
    \sum_{k = 1}^{N} \left(\sqrt{k}\right)^m \leq \frac{M}{\varepsilon \sqrt{2 \sigma}} \left(R \left(\sqrt{N}\right)^{m+1} + \sum_{k = 1}^{N} \left(\sqrt{k}\right)^{m-1}\right), 
$$
for any $m \geq -1, M>0, R >0, \sigma > 0$ and $\varepsilon \leq \frac{M}{\sqrt{2 \sigma}}$,   holds for at least $N = 1$. This means that by choosing $\varepsilon \leq \frac{M}{\sqrt{2 \sigma}} \, (\forall M>, \sigma > 0)$, by Algorithm \ref{alg:MD_func_cons} with \eqref{steps_rulse_alg3}, we have at least one productive step for any $m \geq -1$ and $R > 0 $.   
\end{remark}

Now let us analyze the convergence of Algorithm \ref{alg:MD_func_cons}, by taking the time-varying step size rule \eqref{steps_rulse_alg3}.

Let $M: = \max \{L_F, M_g\}$. By using \eqref{steps_rulse_alg3}, and since $\|F (x^k)\|_* \leq L_F \leq M$ and $\|\nabla g(x^k)\|_* \leq M_g \leq M$, then for any $m >0 $, from Theorem \ref{theo:alg_func_cons}, we have
\begin{align*}
    & \quad \; \operatorname{Gap}(\widehat{x})  = \max_{x \in Q} \left\langle F(x), \widehat{x} - x \right \rangle
    \\& < \frac{\left(\sqrt{2\sigma}\right)^{m}}{M^m \sum_{k \in I} \left(\sqrt{k}\right)^{m}} \Bigg(\frac{R M^{m+1} \left(\sqrt{N}\right)^{m+1} }{\left(\sqrt{2\sigma}\right)^{m+1}} 
    \\& \qquad + \frac{1}{2 \sigma} \sum_{k = 1}^{N} \frac{ M^{m+1} \left(\sqrt{k}\right)^{m-1} }{\left(\sqrt{2\sigma}\right)^{m-1}}
    + M D \sum_{k \in J} \frac{M^m \left(\sqrt{k}\right)^{m}}{\left(\sqrt{2 \sigma}\right)^m}  \Bigg) + \delta
    \\& = \frac{M}{\sqrt{2\sigma}} \cdot \frac{1}{\sum_{k \in I} \left(\sqrt{k}\right)^m} \Bigg( R\left(\sqrt{N}\right)^{m+1} + \sum_{k = 1}^{N} \left(\sqrt{k}\right)^{m-1} 
    \\& \qquad \qquad \qquad \qquad \qquad  \qquad \qquad + \sqrt{2 \sigma} D \sum_{k \in J} \left(\sqrt{k}\right)^{m}  \Bigg) + \delta
    \\& \leq \frac{M}{\sqrt{2 \sigma}} \cdot\frac{1}{\sum_{k \in I} \left(\sqrt{k}\right)^m} \cdot \Bigg( R\left(\sqrt{N}\right)^{m+1} + N \left(\sqrt{N}\right)^{m-1} 
    \\& \qquad \qquad \qquad \qquad \qquad  \qquad \qquad + \sqrt{2\sigma} D |J| \left(\sqrt{N}\right)^{m} \Bigg) + \delta
    \\& = \frac{M(1+ R) \left(\sqrt{N}\right)^{m+1} +\sqrt{2\sigma} M D |J| \left(\sqrt{N}\right)^{m} }{\sqrt{2 \sigma} \sum_{k \in I} \left(\sqrt{k}\right)^m} + \delta.
\end{align*}

Now, by setting 
$$
    \frac{M(1+ R) \left(\sqrt{N}\right)^{m+1} +\sqrt{2\sigma} M D |J| \left(\sqrt{N}\right)^{m} }{\sqrt{2 \sigma} \sum_{k \in I} \left(\sqrt{k}\right)^m} \leq \varepsilon,
$$
and since $|I| \leq N$, we get 
\begin{align*}
    & \quad \; \frac{M(1+ R) \left(\sqrt{N}\right)^{m+1} +\sqrt{2\sigma} M D |J| \left(\sqrt{N}\right)^{m} }{\sqrt{2 \sigma} N\left(\sqrt{N}\right)^m}
    \\& \leq \frac{M(1+ R) \left(\sqrt{N}\right)^{m+1} +\sqrt{2\sigma} M D |J| \left(\sqrt{N}\right)^{m} }{\sqrt{2 \sigma} \sum_{k \in I} \left(\sqrt{k}\right)^m} \leq \varepsilon.
\end{align*}
Thus, 
$$
    \frac{M (1 + R)}{\sqrt{2\sigma} \sqrt{N} } + \frac{M D |J|}{N}  \leq \varepsilon. 
$$
  
Hence, we can formulate the following result.
\begin{corollary}
Let $F: Q \longrightarrow \textbf{E}^*$ be a continuous, bounded, and $\delta$-monotone operator. Let  $g(x) = \max_{1 \leq i \leq p} \{g_i(x)\}$ be an $M_g$-Lipschitz convex function, where $g_i: Q \longrightarrow \mathbb{R},\; \forall i = 1,2, \ldots, p$ are $M_{g_i}$-Lipschitz, and $M_g = \max_{1 \leq i \leq p} \{M_{g_i}\}$. Then for problem, after $N \geq 1$ iterations of Algorithm \ref{alg:MD_func_cons}, such that
\begin{equation}\label{iter_bound_alg3}
    \frac{M (1 + R)}{\sqrt{2\sigma} \sqrt{N} } + \frac{M D |J|}{N}  \leq \varepsilon, 
\end{equation}
for any fixed $m > 0$, with step size rules given in \eqref{steps_rulse_alg3}, it satisfies 
$$
    \operatorname{Gap}(\widehat{x})  = \max_{x \in Q} \left\langle F(x), \widehat{x} - x \right \rangle< \varepsilon + \delta, \quad \text{and} \quad g(\widehat{x}) \leq \varepsilon, 
$$
where $\widehat{x} = \frac{1}{\sum_{k \in I} \left(\gamma_k^f\right)^{-m}} \sum\limits_{k \in I} \left(\gamma_k^f\right)^{-m} x^k$.
\end{corollary}

\bigskip 

Now, by setting $m = 0$ in \eqref{main_ineq_MD_func_cons}, with $\overline{x} = \frac{1}{|I|} \sum_{k \in I} x^k$, we get
\begin{align*}
    & \quad \; \operatorname{Gap}(\overline{x})  = \max_{x \in Q} \left\langle F(x), \overline{x} - x \right \rangle < 
    \\& < \frac{1}{|I|} \left(\frac{R}{\gamma_N} + \sum_{k \in I} \frac{\| F(x^k) \|_*^2 }{2 \sigma} \gamma_k^F + \sum_{k \in J} \frac{\|\nabla g(x^k)\|_*^2 }{2 \sigma} \gamma_k^g    - (\varepsilon - M_g D) |J| \right) + \delta
    \\&  \leq \frac{1}{|I|} \left( \frac{M R \sqrt{N}}{\sqrt{2 \sigma}} + \frac{M}{\sqrt{2 \sigma}} \sum_{k \in I} \frac{1}{\sqrt{k}} +\frac{M}{\sqrt{2 \sigma}} \sum_{k \in J} \frac{1}{\sqrt{k}} + MD |J| \right) + \delta
    \\& = \frac{M}{|I| \sqrt{2\sigma}} \left(R  \sqrt{N} + \sum_{k = 1}^{N} \frac{1}{\sqrt{k}}\right) + \frac{MD|J|}{|I|} + \delta
    \\& \leq \frac{M\sqrt{N}}{|I| \sqrt{2\sigma}} \left( 2 + R\right) + \frac{MD|J|}{|I|} + \delta.
\end{align*}

Thus, by setting $\frac{M\sqrt{N}}{|I| \sqrt{2\sigma}} \left( 2 + R\right) + \frac{MD|J|}{|I|}  \leq \varepsilon$ and since $|I| \leq N$, we get 
$$
    \frac{M\left( 2 + R\right)}{\sqrt{2\sigma} \sqrt{N} }  + \frac{MD|J|}{N} \leq \varepsilon.
$$

Hence, for $m=0$,  we can formulate the following result.

\begin{corollary}
Let $F: Q \longrightarrow \textbf{E}^*$ be a continuous, bounded, and $\delta$-monotone operator. Let  $g(x) = \max_{1 \leq i \leq p} \{g_i(x)\}$ be an $M_g$-Lipschitz convex function, where $g_i: Q \longrightarrow \mathbb{R},\; \forall i = 1,2, \ldots, p$ are $M_{g_i}$-Lipschitz, and $M_g = \max_{1 \leq i \leq p} \{M_{g_i}\}$. Then, after $N \geq 1$ iterations of Algorithm \ref{alg:MD_func_cons}, such that
$$
    \frac{M\left(2 + R\right)}{\sqrt{2\sigma} \sqrt{N} }  + \frac{MD|J|}{N} \leq \varepsilon,
$$
with $m = 0$ and  step size rules given in \eqref{steps_rulse_alg3}, it satisfies 
$$
    \operatorname{Gap}(\overline{x})  = \max_{x \in Q} \left\langle F(x), \overline{x} - x \right \rangle< \varepsilon + \delta, \quad \text{and} \quad g(\overline{x}) \leq \varepsilon, 
$$
where $\overline{x} = \frac{1}{|I|} \sum_{k \in I} x^k$.
\end{corollary}

\begin{remark}
By setting $m = -1$ in \eqref{main_ineq_MD_func_cons}, with $\widetilde{x} = \frac{1}{\sum\limits_{k \in I} \frac{1}{\sqrt{k}}} \sum\limits_{k \in I}\frac{1}{\sqrt{k}} x^k$, we have
\begin{align*}
    \operatorname{Gap}(\widetilde{x})&<  \frac{M}{\sqrt{2 \sigma}} \cdot \frac{1}{\sum_{k \in I} \sqrt{k}} \cdot \left( R + \sum_{k = 1}^{N} \frac{1}{k} + D \sqrt{2\sigma} \sum_{k \in J} \frac{1}{\sqrt{k}}\right)  + \delta
    \\& \leq\frac{M}{\sqrt{2 \sigma}} \cdot \frac{1}{2 \sqrt{|I| + 1 } - 2}  \left( 1+R + \log (N) +D \sqrt{2\sigma} \sum_{k \in J} \frac{1}{\sqrt{k}} \right) + \delta
    \\& \leq \frac{M  }{ \sqrt{\sigma} \sqrt{|I|}} \left(1 +R +  \log (N) +D \sqrt{2\sigma} \sum_{k \in J} \frac{1}{\sqrt{k}} \right) + \delta. 
\end{align*}

Thus, after $N \geq 1$ iterations of Algorithm \ref{alg:MD_func_cons} (with $m = -1$), such that
$$
    \frac{M  }{ \sqrt{\sigma} \sqrt{|I|}} \left( 1+ R + \log (N) +D \sqrt{2\sigma} \sum_{k \in J} \frac{1}{\sqrt{k}} \right) \leq \varepsilon, 
$$
it satisfies,
$$
    \operatorname{Gap}(\widetilde{x})  = \max_{x \in Q} \left\langle F(x), \overline{x} - x \right \rangle< \varepsilon + \delta, \quad \text{and} \quad g(\widetilde{x}) \leq \varepsilon.
$$
\end{remark}

\section{Numerical experiments}\label{sect_numerical}
To show the advantages and effects of the considered weighting scheme for generated points by Algorithm \ref{alg_mirror_descent} (see Theorem \ref{theo_main_ineq_mirror_desc}) in its convergence, a series of numerical experiments were performed for some examples of the classical variational inequality problem. We compare the performance of Algorithm \ref{alg_mirror_descent} with the Modified Projection Method (MPM) proposed in  \cite{Khanh2014Modified}. In our experiments, we take the standard Euclidean prox-structure, namely $\psi(x) = \frac{1}{2} \|x\|_2^2$ which is $1$-strongly functions (i.e., $\sigma = 1$) and the corresponding Bregman divergence is $V_{\psi}(x, y) = \frac{1}{2} \|x - y\|_2^2$. In all experiments, we take the set $Q$ as a unit ball in $\mathbb{R}^n$ with the center at $\textbf{0} \in \mathbb{R}^n$. The compared methods start from the same initial point $x^1 = \left(\frac{1}{\sqrt{n}}, \ldots, \frac{1}{\sqrt{n}}\right) \in \mathbb{R}^n$. The results of the comparisons for the considered examples are presented in Figs. \ref{fig_ex1_ex2} and \ref{fig_ex4}. These results show the values $\|F_k / F(x^1)\|_2^2$, where  $F_k : = F(\widehat{x}_k)$, and $\widehat{x}_k = \frac{1}{\sum_{i = 1}^{k} \gamma_i^{-m}} \sum_{i = 1}^{k} \gamma_i^{-m} x^i$.

\begin{example}\label{example_1} (\cite{Dong2018Inertial})
Let $F: \mathbb{R}^2 \longrightarrow \mathbb{R}^2$ be a monotone and bounded operator in the unit ball, defined as follows
\begin{equation}
    F(x_1, x_2) = \left(2x_1 + 2x_2 + \sin(x_1), -2x_1 + 2 x_2 + \sin(x_2)\right).  
\end{equation}
\end{example}

\begin{example}\label{example_3} (\cite{Sahu2021Inertial})
Let $F: \mathbb{R}^3 \longrightarrow \mathbb{R}^3$ be a monotone and bounded operator in the unit ball, defined as follows
\begin{equation}
    F(x_1, x_2, x_3) = \left(F_1(x_1, x_2, x_3), F_2(x_1, x_2, x_3), F_3(x_1, x_2, x_3)\right), 
\end{equation}
where $r, s , t \in \mathbb{R}$ and $F_1(x_1, x_2, x_3) = x_1 - s x_2 +t x_3 + \sin(x_1), F_2(x_1, x_2, x_3) = x_2 - r x_3 +s x_1 + \sin(x_2), F_3(x_1, x_2, x_3) = x_3 - t x_1 +r x_2 + \sin(x_3).$
\end{example}

The results for Examples \ref{example_1} and \ref{example_3}, presented in Fig. \ref{fig_ex1_ex2}. From this figure, we can see that MPM works better than Algorithm \ref{alg_mirror_descent} only for small values of the parameter $m$. But Algorithm \ref{alg_mirror_descent} works better than MPM, with a big difference between their performance when we increase the value of the parameter $m$.  

\begin{figure}[htp]
\minipage{0.50\textwidth}
\includegraphics[width=\linewidth]{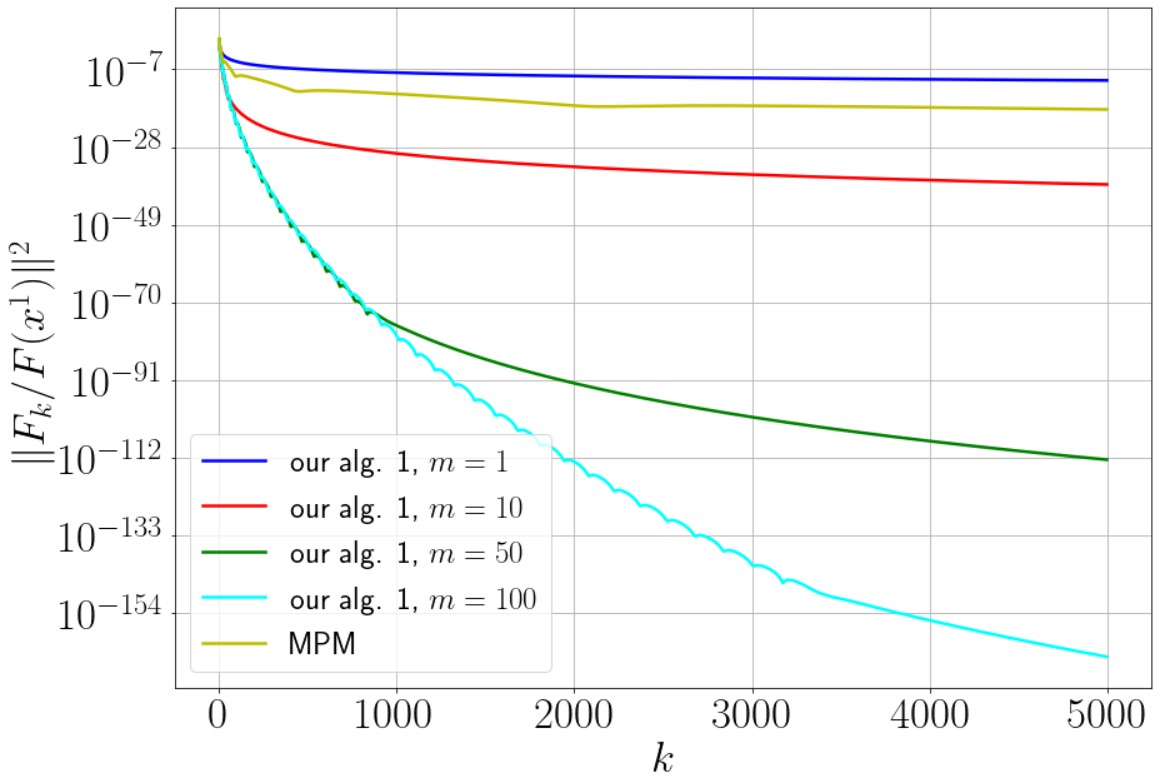}
\endminipage
\minipage{0.50\textwidth}
\includegraphics[width=\linewidth]{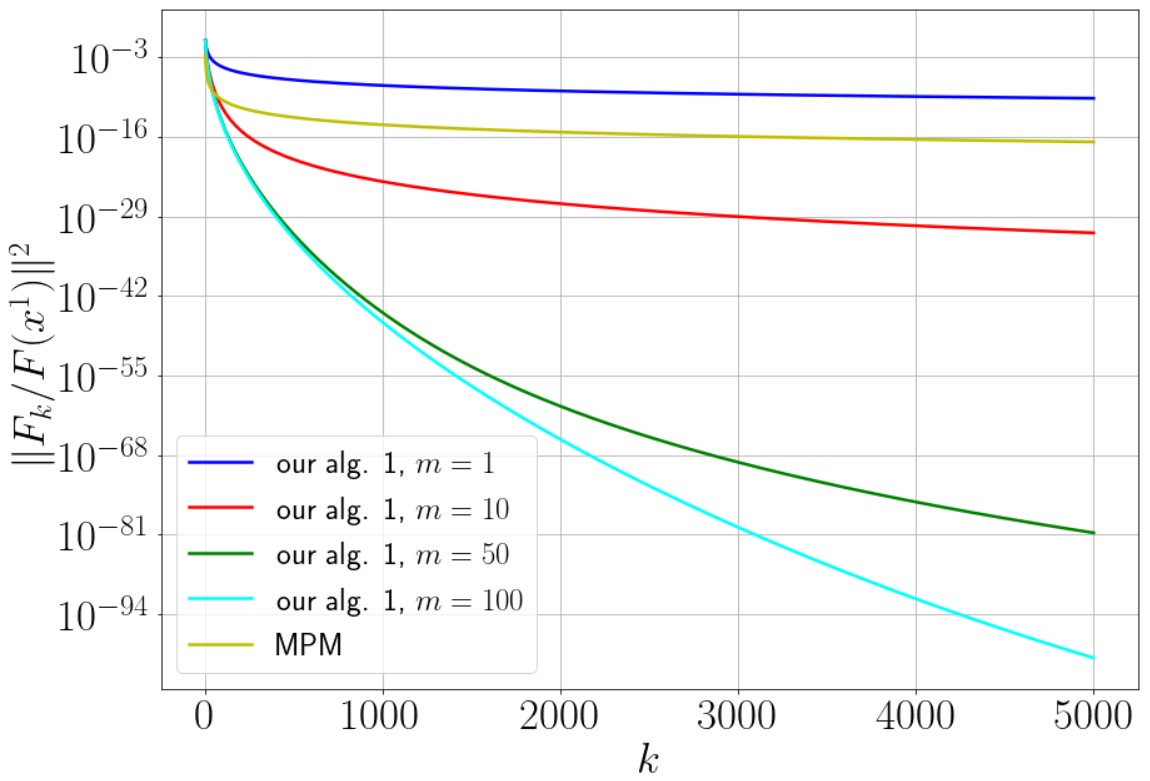}
\endminipage
\caption{Results of  Algorithm \ref{alg_mirror_descent} and Modified Projection Method \cite{Khanh2014Modified}, for Example \ref{example_1} (in the left) and for Example \ref{example_3} (in the right). }
\label{fig_ex1_ex2}
\end{figure} 

\begin{example}\label{example_4}
In this example, we consider the HpHard (or Harker-Pang) problem \cite{Qu2024extra}. Let $F: \mathbb{R}^n \longrightarrow \mathbb{R}^n$ be an operator defined by
\begin{equation}
    F(x) = K x + q, \quad K = A A^{\top } + B + C,  q \in \mathbb{R}^n,
\end{equation}
where $A \in \mathbb{R}^{n \times n}$ is a matrix, $B \in \mathbb{R}^{n \times n}$ is a skew-symmetric matrix ($A$ and $B$ are randomly generated from  a normal (Gaussian) distribution with mean equals $0$ and scale equals $0.01$) and $C \in \mathbb{R}^{n \times n}$ is a diagonal matrix with non-negative diagonal entries (randomly generated from the continuous uniform distribution over the interval $[0, 1)$). Therefore, it follows that $K$ is positive semidefinite. The operator $F$ is monotone and bounded in the unit ball with constant $L_F = \|K\|_2 + \|q\|_2$. For $q = \textbf{0} \in \mathbb{R}^n$, the solution of problem \eqref{main_constrained_prob}, is $x^* = \textbf{0} \in \mathbb{R}^n$. 
\end{example}

The results for Example \ref{example_4}, presented in Fig. \ref{fig_ex4}. From this figure we can see that Algorithm \ref{alg_mirror_descent} always works better than MPM for any $m \geq 1$. 

\begin{figure}[htp]
\minipage{0.95\textwidth}
\includegraphics[width=\linewidth]{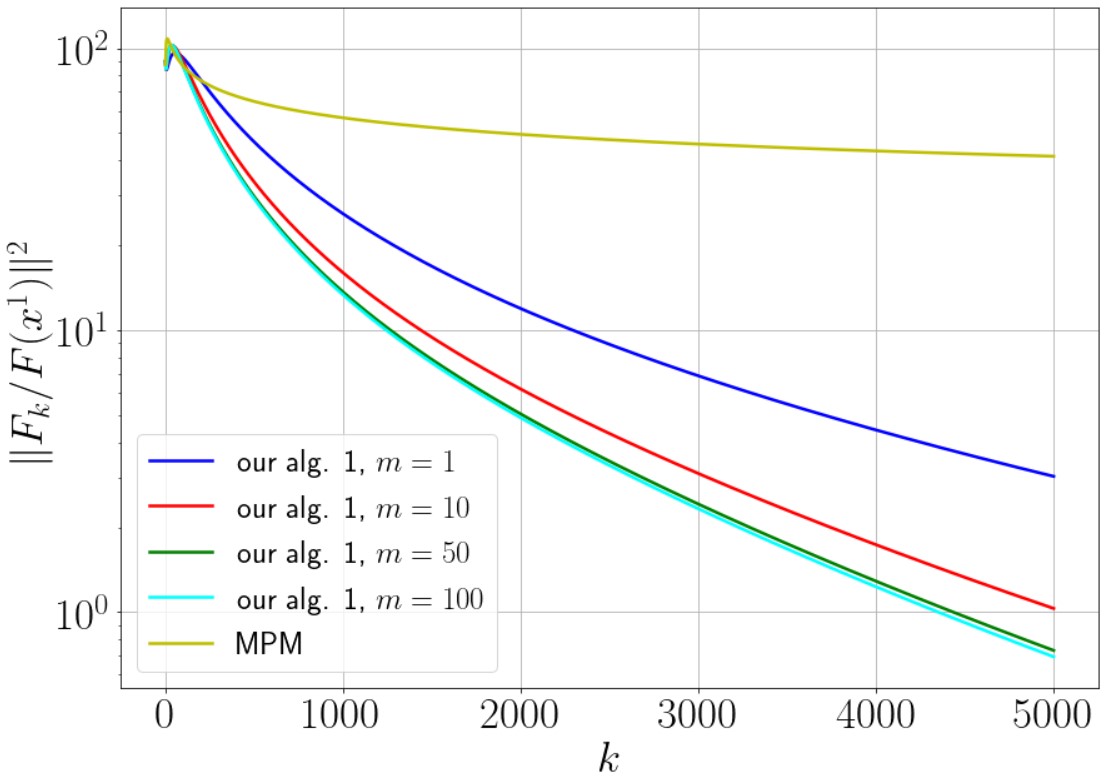}
\endminipage
\caption{Results of  Algorithm \ref{alg_mirror_descent} and Modified Projection Method \cite{Khanh2014Modified}, for Example \ref{example_4} with $n=100$. }
\label{fig_ex4}
\end{figure}

\section{Conclusion}\label{sec_conclusions}
In this paper, we studied two classes of variational inequality problems. The first is classical constrained (i.e., without functional constraints) variational inequality and the second is the same problem with functional constrained (inequality type constraints). To solve such problems, we proposed mirror descent-type methods with a weighting scheme for the generated points in each iteration of the algorithms. For the second class of problems, we proposed a mirror descent method by switching between productive and non-productive steps. We analyzed the proposed methods for the time-varying step sizes and proved the optimal convergence rate of the proposed algorithm concerning the classical variational inequality problems with bounded and $\delta$-monotone operators. We conducted some numerical experiments, which illustrate the advantages of the presented weighting scheme for some examples of the classical variational inequality problem, with a comparison with the modified projection method. As a future work, there are many directions connecting with the problems under consideration, such as the results for the Lipschitz monotone and strongly monotone operators, as well as for the stochastic setting of the problem.    

\end{document}